\documentclass[reqno,12pt]{amsart}
\usepackage{amsmath,amssymb,color}
\usepackage{amsthm}
\usepackage{comment}
\usepackage{graphicx}

\newtheorem{theorem}{Theorem}

\newtheorem{proposition}{Proposition}

\def\re{\mathbb{R}}

\def\Sp{\mathbb{S}}

\def\pd{\partial}
\def\ol{\overline}

\def\la{\lambda}

\def\({\left(}
\def\){\right)}

\def\pd{\partial}

\def\|{\Vert}

\begin{document}
\title[Exact solution structures]{Exact solution structures on some nonlocal overdetermined problems}

\author{Kazuki Sato}

\address{
Department of Mathematics, Osaka Metropolitan University \\
3-3-138, Sumiyoshi-ku, Sugimoto-cho, Osaka, Japan \\
}

\email{sf22817a@st.omu.ac.jp \\}

\author{Futoshi Takahashi}

\address{
Department of Mathematics, Osaka Metropolitan University \\
3-3-138, Sumiyoshi-ku, Sugimoto-cho, Osaka, Japan \\
}

\email{futoshi@omu.ac.jp \\}

\begin{abstract}
In this paper, we study the solution structures of Serrin-type overdetermined problems with Kirchhoff-type nonlocal terms.
We prove that the exact number of solutions is the same as those of some transcendental equations defined by the nonlocal terms. 
We also obtain the explicit form of solutions by using the unique solutions of the overdetermined problems without the nonlocal terms. 
\end{abstract}

%34C23: Bifurcation
%37G99: Local and nonlocal bifurcation theory (none of the above)
% 26D10: Inequalities involving derivatives and differential and integral operators
% 26D15: Inequalities for sums, series and integrals
% 35A23: Inequalities involving derivatives and differential and integral operators, inequalities for integrals
% 46E35: Sobolev spaces and other spaces of ``smooth" functions, embedding theorems, trace theorems

\date{\today}

\dedicatory{}

\maketitle

\section{Introduction}

Let $\Omega \subset \re^N (N \ge 1)$ be a bounded domain with $C^2$ boundary and $k \in \{ 1, 2, \cdots, N \}$.
In this paper, we consider the following fully nonlinear overdetermined problem with Kirchhoff-type nonlocal term:
\begin{equation}
\label{OD_HM}
	\begin{cases}
	&M\(\| u \|_{L^{p}(\Omega)}, \| \nabla u \|_{L^{q}(\Omega)} \) S_k(D^2 u) = \la \ \text{in} \ \Omega, \\
	&u = 0 \quad \text{on} \ \pd\Omega, \\
	&\frac{\pd u}{\pd \nu} = c > 0 \quad \text{on} \ \pd\Omega,
	\end{cases}
\end{equation}
where $0 < p, q \le \infty$, $\la > 0$ are given constants,
$c$ is an {\it unknown} positive constant, $\nu$ is the outer unit normal to $\pd\Omega$,
and $M(s, t)$ is a positive function in $(s, t) \in \re_+ \times \re_+$.
We do not assume any continuity for $M$.
For a real symmetric matrix $A$, let $S_k(A)$ denote the $k$-th elementary symmetric function of the eigenvalues of $A$ (counted with multiplicity):
\[
	S_k(A) = \sum_{1 \le i_1 < i_2 < \cdots < i_k \le N} \la_{i_1} \la_{i_2} \cdots \la_{i_k}.
\]
Thus $S_1(A) = \sum_{i=1}^N \la_i$ and $S_N(A) = \prod_{i=1}^N \la_i$.
Finally, let $D^2 u(x) = \( \frac{\pd^2 u}{\pd x_i \pd x_j}(x) \)_{1 \le i, j \le N}$ denote the Hessian matrix of a $C^2$-function $u$.
Then $S_k(D^2 u)$ is called the {\it $k$-Hessian operator}.
Note that $S_1(D^2 u) = \Delta u$ and $S_N(D^2 u) = {\rm det } (D^2 u)$.

In a paper \cite{BNST}, the authors consider the Serrin-type overdetermined problem for the $k$-Hessian operator:
\begin{equation}
\label{OD_H}
	\begin{cases}
	&S_k(D^2 u) = \binom{N}{k} \ \text{in} \ \Omega, \\
	&u = 0 \quad \text{on} \ \pd\Omega, \\
	&\frac{\pd u}{\pd \nu} = c > 0 \quad \text{on} \ \pd\Omega,
	\end{cases}
\end{equation}
where $c > 0$ is unknown, $\binom{N}{k} = \frac{N !}{k! (N-k)!}$ for $k \in \{1,2,\dots, N \}$.
They prove that if the problem \eqref{OD_H} admits a solution $u \in C^2(\ol{\Omega})$ for some $k \in \{1,2,\dots, N \}$,
then, up to a translation, $\Omega$ must be a ball, say $\Omega = B_R(x_0)$ for $R > 0$ and $x_0 \in \re^N$,
and the solution is of the form 
\begin{equation}
\label{U_R}
	u(x) = U_{R, x_0}(x) :=  \frac{|x-x_0|^2 - R^2}{2}
\end{equation}
with $c = R$. 
Note that the $k$-Hessian operator $S_k(D^2 u)$ is fully nonlinear and in general is not elliptic.
In spite of these difficulties, the authors in \cite{BNST} give a shorter alternative proof which does not exploit maximum principles directly 
and to extend the famous result by Serrin \cite{Serrin} (for $k=1$) in this setting.
Their method uses a Pohozaev-type identity by Tso \cite{Tso} and reminds us of an alternative proof by Weinberger \cite{Weinberger}.

Recently, the Serrin-type overdetermined problem with Kirchhof-type nonlocal term:
\[
	\begin{cases}
	&M\(\| u \|_{L^{p}(\Omega)}, \| \nabla u \|_{L^{q}(\Omega)} \) \Delta u = \la \ \text{in} \ \Omega, \\
	&u = 0 \quad \text{on} \ \pd\Omega, \\
	&\frac{\pd u}{\pd \nu} = c \quad \text{on} \ \pd\Omega,
	\end{cases}
\]
where $0 < p, q \le \infty$, $\la > 0$ are given constants, $c$ is an unknown constant, and $M: \re_+ \times \re_+ \to \re_+$ is a positive function,
has been considered in \cite{Sato-TF}, 
and the exact number of solutions according to the values of the bifurcation parameter $\la > 0$ is determined 
by a transcendental equation defined by the nonlocal term for a real unknown variable. 

In this paper, we extend the main result in \cite{Sato-TF} to the problem \eqref{OD_HM}. 
Original argument is coming from \cite{ACM}.
\begin{theorem}
\label{Theorem:H}
Let $M : \re_+ \times \re_+ \to \re_+$, $0 < p, q \le \infty$, $\la > 0$, and $k \in \{ 1, 2, \dots, N \}$.
Then if \eqref{OD_HM} admits a solution $u \in C^2(\ol{\Omega})$ for an unknown constant $c > 0$,
then $\Omega$ must be a ball and $u$ must be radially symmetric with respect to the center of the ball. 

Let $\Omega = B_R(x_0)$ for $R > 0$ and $x _0 \in \re^N$ and consider the system of equations
with respect to $(s, t) \in \re_+ \times \re_+$:
\begin{equation}
\label{System}
	\begin{cases}
	&s = \left\{ \frac{M(s, t)}{\la} \binom{N}{k} \right\}^{-1/k} \| U_{R, x_0} \|_{L^{p}(B_R(x_0))} \\
	&t = \left\{ \frac{M(s, t)}{\la} \binom{N}{k} \right\}^{-1/k} \| \nabla U_{R, x_0} \|_{L^{q}(B_R(x_0))}
	\end{cases}
\end{equation}
where $U_{R, x_0}$ defined in \eqref{U_R} is the unique solution of \eqref{OD_H} for $\Omega = B_R(x_0)$.
Then for $\la > 0$, the problem \eqref{OD_HM} for $\Omega = B_R(x_0)$ has the same number of solutions of the system of equations \eqref{System}.
Also define
\begin{equation}
\label{g(s)}
	g(s) =  s^k M\(s, \frac{\| \nabla U_{R, x_0} \|_{L^{q}(B_R(x_0))}}{\| U_{R, x_0} \|_{L^{p}(B_R(x_0)}} s \) \binom{N}{k}
\end{equation}
for $s > 0$.
Then the number of solutions of \eqref{System} is the same as the number of solutions to the equation
\begin{equation}
\label{Single}
	g(s) = \la \| U_{R, x_0} \|_{L^{p}(B_R(x_0))}^k
\end{equation}
with respect to $s > 0$.
Moreover, any solution $u_{\la}$ of \eqref{OD_HM} has the form 
\[
	u_{\la}(x) = s_* \frac{U_{R,x_0}(x)}{\| U_{R, x_0} \|_{L^{p}(B_R(x_0))}}
\]
where $s_*$ is any solution of \eqref{Single}.
\end{theorem}

\vspace{1em}
In the second part, we study an exterior overdetermined problem with Kirchhoff-type nonlocal term as described below:
Let $D \subset \re^N$ $(N \ge 2)$ be a bounded $C^2$-domain containing the origin as an interior point.
We consider
\begin{equation}
\label{OD_EM}
	\begin{cases}
	&M\(\| u \|_{L^{p}(\Omega)}, \| \nabla u \|_{L^{q}(\Omega)} \) \Delta u(x) = \la |x|^{-N-2}, \quad x \in \Omega, \\
	&u(x) = 0, \quad x \in \pd \Omega, \\
	&|\nabla u(x)| = c |x|^{-N} > 0, \quad x \in \pd \Omega, \\
	&|u(x)| = o(1) \quad (|x| \to \infty), \quad (\text{if} \ N \ge 3), \\
	&|u(x)| = O(1) \quad (|x| \to \infty), \quad (\text{if} \ N = 2),
	\end{cases}
\end{equation}
where $\Omega = \re^N \setminus \ol{D}$ is an exterior domain,
$0 < p, q \le \infty$, $\la > 0$ are given constants, $c$ is an unknown positive constant, 
and $M(s, t)$ is a positive function in $(s, t) \in \re_+ \times \re_+$.

Let $B$ be the closed unit ball in $\re^N$ for $N \ge 2$ and define
\begin{equation}
\label{U}
	U(x) = \frac{1}{2} \( |x|^{-N} - |x|^{2-N}\) \quad \text{for} \ |x| > 1.
\end{equation}
Then a direct computation shows that $U$ is a solution of
\[
	\begin{cases}
	&\Delta U = N |x|^{-N-2} \ \text{in} \ B^c = \re^N \setminus B, \\
	&U = 0 \quad \text{on} \ \pd B^c, \\
	&|\nabla U| = 1 \quad \text{on} \ \pd B^c, \\
	&|U(x)| \to 0 \quad (|x| \to \infty), \quad (\text{if} \ N \ge 3), \\
	&|U(x)| \to \frac{1}{2} \quad (|x| \to \infty), \quad (\text{if} \ N = 2).
	\end{cases}
\]
Also note that $U \in L^p(B^c)$ if and only if $\frac{N}{N-2} < p \le \infty$ if $N \ge 3$, $p = \infty$ if $N = 2$, 
and $|\nabla U| \in L^q(B^c)$ if and only if $\frac{N}{N-1} < q \le \infty$ if $N \ge 3$, $\frac{2}{3} < q \le \infty$ if $N = 2$.
In this case, we compute 
\begin{align*}
	&\| U \|_{L^p(B^c)}^p = (\frac{1}{2})^{p+1} |\Sp^{N-1}| B\(\tfrac{p(N-2)-N}{2}, p+1\), \quad (\frac{N}{N-2} < p < \infty, \ N \ge 3), \\
	&\| U \|_{L^\infty(B^c)} = \(\frac{1}{N-2}\)\(\frac{N}{N-1}\)^{-N/2}, \quad (N \ge 3), \\
	&\| U \|_{L^\infty(B^c)} = \frac{1}{2}, \quad (N = 2), \\
	&\| \nabla U \|_{L^q(B^c)}^q = (\frac{1}{2})^{q+1} |\Sp^{N-1}| N^q \(\frac{N-2}{N}\)^{\tfrac{q(N+1)-N}{2}} \times \\
	&\(B \(\tfrac{q(N-1)-N}{2}, q+1\) + B_{\tfrac{2}{N}}\( q+1, \tfrac{-(N+1)q + N}{2} \) \), \quad (\frac{N}{N-1} < q < \infty, \ N \ge 3), \\
	&\| \nabla U \|_{L^q(B^c)}^q = \frac{2\pi}{3q-2} \quad (\frac{2}{3} < q < \infty, \ N = 2), \\
	&\| \nabla U \|_{L^\infty(B^c)} = \(\frac{N}{N-1}\)\(\frac{N(N+1)}{(N-1)(N-2)}\)^{-(N+1)/2}, \quad (N \ge 3), \\
	&\| \nabla U \|_{L^\infty(B^c)} = 1, \quad (N = 2),
\end{align*}
where $B(x, y) = \int_0^1 t^{x-1} (1-t)^{y-1} dt$ denotes the Beta function, and 
$B_z(x, y) = \int_0^z t^{x-1} (1-t)^{y-1} dt$ denotes the incomplete Beta function for $z \in (0,1)$.

Assume $\frac{N}{N-2} < p \le \infty$ if $N \ge 3$, $p = \infty$ if $N = 2$, and $\frac{N}{N-1} < q \le \infty$ if $N \ge 3$, $\frac{2}{3} < q \le \infty$ if $N=2$. 
Let us consider the system of equations 
\begin{equation}
\label{System_E}
	\begin{cases}
	&s = \left\{ \frac{M(s, t)}{\la} N \right\}^{-1} \| U \|_{L^{p}(B^c)} \\
	&t = \left\{ \frac{M(s, t)}{\la} N \right\}^{-1} \| \nabla U \|_{L^{q}(B^c)}
	\end{cases}
\end{equation}
with respect to $(s, t) \in \re_+ \times \re_+$, where $U$ is in \eqref{U}.
Then we have the following:

\begin{theorem}
\label{Theorem:E}
Let $N \ge 2$, $M : \re_+ \times \re_+ \to \re_+$, and
\begin{align*}
	&\frac{N}{N-2} < p \le \infty \quad \text{and} \quad \frac{N}{N-1} < q \le \infty, \quad \text{if} \ N \ge 3, \\ 
	&p = \infty \quad \text{and} \quad \frac{2}{3} < q \le \infty,  \quad \text{if} \ N = 2.
\end{align*}
Assume that \eqref{OD_EM} admits a solution $u \in C^2(\ol{\Omega})$ for an unknown constant $c > 0$,
where $\Omega = \re^N \setminus \ol{D}$ is an exterior domain.
Then $D$ must be a ball and $u$ must be radially symmetric with respect to the center of the ball. 

Assume, without loss of generality, that $\ol{D} = B$.
Then for $\la > 0$, the problem \eqref{OD_EM} for $\Omega = B^c$ has the same number of solutions of the system of equations \eqref{System_E}.
Also define
\[
	g(s) =  s M\(s, \frac{\| \nabla U \|_{L^{q}(B^c)}}{\| U \|_{L^{p}(B^c)}} s \) N
\]
for $s > 0$.
Then the number of solutions of \eqref{System_E} is the same as the number of solutions to the equation
\begin{equation}
\label{Single_O}
	g(s) = \la \| U \|_{L^{p}(B^c)}
\end{equation}
with respect to $s > 0$.
Moreover, any solution $u_{\la}$ of \eqref{OD_EM} has the form 
\[
	u_{\la}(x) = s_* \frac{U(x)}{\| U \|_{L^{p}(B^c)}}
\]
where $s_*$ is any solution of \eqref{Single_O}.
\end{theorem}

\section{Proof of Theorem \ref{Theorem:H}.}

In this section, we prove Theorem \ref{Theorem:H}.

\begin{proof}
First assume that there exists a solution $u \in C^2(\ol{\Omega})$ of \eqref{OD_HM} and put 
\[
	v = \gamma u
\] 
where $\gamma > 0$ is chosen so that % $\gamma^k \frac{\la}{M(\| u \|_{L^p}(\Omega), \| \nabla u \|_{L^q(\Omega)})} = \binom{N}{k}$, that is,
\[
	\gamma = \( \frac{M(\| u \|_{L^{p}(\Omega)}, \| \nabla u \|_{L^{q}(\Omega)})}{\la} \binom{N}{k} \)^{1/k}.
\]
Note that the $k$-Hessian operator is homogeneous of degree $k$:
Then
\begin{align*}
	S_k(D^2 v) = S_k (\gamma D^2 u) &= \gamma^k S_k(D^2 u) \\
	&\overset{\eqref{OD_HM}}{=} \gamma^k \dfrac{\la}{M(\| u \|_{L^p(\Omega)}, \| \nabla u \|_{L^q(\Omega)})} = \binom{N}{k}
\end{align*}
in $\Omega$.
Also we see $v = 0$ on $\pd\Omega$ and $\frac{\pd v}{\pd \nu} = \gamma c = const.$ on $\pd\Omega$.
Thus by the result of \cite{BNST}, we see that $\Omega$ must be a ball, say $\Omega = B_R(x_0)$ for some $R > 0$ and $x _0 \in \re^N$ and 
$v \equiv U_{R,x_0}(x)$. 
This implies that 
\begin{equation}
\label{u_form}
	u(x) = \gamma^{-1} U_{R,x_0}(x), \quad \nabla u(x) = \gamma^{-1} \nabla U_{R,x_0}(x).
\end{equation}
Define
\begin{equation}
\label{st}
	s = \| u \|_{L^{p}(B_R(x_0))} \quad \text{and} \quad t = \| \nabla u \|_{L^{q}(B_R(x_0))}.
\end{equation}
Then by \eqref{u_form}, we have
\[
	\begin{cases}
	&s = \gamma^{-1} \| U_{R,x_0} \|_{L^{p}(B_R(x_0))}, \\
	&t = \gamma^{-1} \| \nabla U_{R,x_0} \|_{L^{q}(B_R(x_0))},
	\end{cases}
\]
which is equivalent to \eqref{System}.
This shows that 
\[
	(s, t) = (\| u \|_{L^{p}(B_R(x_0))}, \| \nabla u \|_{L^{q}(B_R(x_0))})
\]
is a solution to \eqref{System} and thus
\[
	\sharp \{ u : \text{solutions of \eqref{OD_HM} for $\Omega = B_R(x_0)$} \} \le \sharp \{ (s, t) \in (\re_+)^2 : \text{solutions of \eqref{System}} \},
\]
where $\sharp A$ denotes the cardinality of the set $A$.

On the other hand, let $\Omega = B_R(x_0)$ for some $R > 0$ and $x_0 \in \re^N$ 
and let $(s, t) \in (\re_+)^2$ be any solution to \eqref{System}.
Note that by \eqref{System}, we see
\begin{align}
\label{System2}
	&\frac{s}{\| U_{R, x_0} \|_{L^{p}(B_R(x_0))}} = \frac{t}{\| \nabla U_{R,x_0} \|_{L^{q}(B_R(x_0))}} = \( \frac{M(s, t)}{\la} \binom{N}{k} \)^{-1/k}.
\end{align}
Thus if we define 
\begin{equation}
\label{u_def}
	u(x) = s \frac{U_{R, x_0}(x)}{\| U_{R, x_0} \|_{L^{p}(B_R(x_0))}}\(= t \frac{U_{R, x_0}(x)}{\| \nabla U_{R, x_0} \|_{L^{q}(B_R(x_0))}} \),
\end{equation}
then we have $u = 0$, $\frac{\pd u}{\pd \nu} = const.$ on $\pd B_R(x_0)$ and \eqref{st} holds by \eqref{u_def}.
Moreover, we have
\begin{align*}
	&M\(\| u \|_{L^{p}(B_R(x_0))}, \| \nabla u \|_{L^{q}(B_R(x_0))}\) S_k(D^2 u) \overset{\eqref{st}}{=} M(s, t) S_k(D^2 u) \\
	&\overset{\eqref{u_def}}{=} M(s, t) \( \frac{s}{\| U_{R, x_0} \|_{L^{p}(B_R(x_0))}} \)^k \underbrace{S_k(D^2 U_{R, x_0})}_{=\binom{N}{k}} \\
	&\overset{\eqref{System2}}{=} M(s, t) \( \frac{M(s,t)}{\la} \binom{N}{k} \)^{-1} \binom{N}{k} = \la
\end{align*}
on $B_R(x_0)$.
This shows that
\[
	\sharp \{ u : \text{solutions of \eqref{OD_HM} for $\Omega = B_R(x_0)$} \} \ge \sharp \{ (s, t) \in (\re_+)^2 : \text{solutions of \eqref{System}} \}.
\]
Thus the number of solutions of \eqref{OD_HM} for $\Omega = B_R(x_0)$ and that of \eqref{System} are the same.

Also by \eqref{System2}, we can rewrite the system of equations \eqref{System} into a single equation for $s$: 
\[
	s = \( \dfrac{M\(s, \dfrac{\| \nabla U_{R, x_0} \|_{L^{q}(B_R(x_0))}}{\| U_{R, x_0} \|_{L^{p}(B_R(x_0))}}\) s}{\la} \binom{N}{k} \)^{-1/k} \| U_{R, x_0} \|_{L^p(B_R(x_0)}, 
\]
which is equivalent to \eqref{Single} with $g(s)$ in \eqref{g(s)}.
Thus the number of solutions of \eqref{System} for $\Omega = B_R(x_0)$ and that of \eqref{Single} are also the same. 
\end{proof}

\section{Exterior overdetermined problem.}

Let $D \subset \re^N$, $N \ge 2$ be a bounded $C^2$-domain containing the origin as an interior point,
and put $\Omega = \re^N \setminus \ol{D}$.
First, we consider an exterior overdetermined problem without Kirchhoff term: 
\begin{equation}
\label{OD_E}
	\begin{cases}
	&\Delta u = N |x|^{-N-2} \ \text{in} \ \Omega, \\
	&u = 0 \quad \text{on} \ \pd\Omega, \\
	&|\nabla u| = c|x|^{-N} \quad \text{on} \ \pd\Omega, \\
	&|u(x)| = o(1) \quad (|x| \to \infty), \quad (\text{if} \ N \ge 3), \\
	&|u(x)| = O(1) \quad (|x| \to \infty), \quad (\text{if} \ N = 2),
	\end{cases}
\end{equation}
where $c$ is an unknown positive constant. 
Direct computation shows that $U(x)$ in \eqref{U} is the exact solution of \eqref{OD_E} for $\Omega = B^c$ with $c = 1$.
On the uniqueness, we have the following proposition.

\begin{proposition}
\label{Prop:E}
Let $N \ge 2$ and $\Omega = \re^N \setminus \ol{D}$ be an exterior domain.
If the problem \eqref{OD_E} admits a solution $u \in C^2(\ol{\Omega})$, then $D$ must be a ball and $u$ must be radially symmetric with respect to the center of the ball.
Let $\Omega = B^c = \re^N \setminus B$. Then $U(x)$ in \eqref{U} is the unique solution of \eqref{OD_E} and $c$ must satisfy $c = 1$.
\end{proposition}

\begin{proof}
First we treat $N \ge 3$ and assume that $u \in C^2(\ol{\Omega})$ is a solution of \eqref{OD_E}.
Define a function $w: \Omega^* \setminus \{ 0 \} \to \re$ such that
\begin{equation}
\label{Kelvin}
	w(y) = |y|^{2-N} u\( \frac{y}{|y|^2}\), \quad (y \in \Omega^* \setminus \{ 0 \}),
\end{equation}
where
\[
	\Omega^* = \{ y = \frac{x}{|x|^2} \ : \ x \in \Omega \} \cup \{ 0 \},
\]
that is, $w$ is the Kelvin transform of $u$.
Note that $x = \frac{y}{|y|^2} \in \Omega$ is equivalent to $y = \frac{x}{|x|^2} \in \Omega^* \setminus \{ 0 \}$.
By a well-known formula for the Kelvin transform and the equation in \eqref{OD_E}, we have
\begin{align*}
	\Delta_y w(y) = |y|^{-2-N} (\Delta_x u) \(\frac{y}{|y|^2}\) = |x|^{2+N} \Delta_x u(x) = N
\end{align*}
for $y \in \Omega^* \setminus \{ 0 \}$.
Since $y \in \pd \Omega^*$ is equivalent to $x \in \pd \Omega$,
$w(y) \Big|_{y \in \pd \Omega^*} = u(x) \Big|_{x \in \pd \Omega} = 0$.
Furthermore, a direct computation gives that
\begin{align*}
	\nabla_y w(y) = |y|^{-N} \left\{ (2-N) u(x) y - 2 (\nabla_x u(x) \cdot x) y + \nabla_x u(x) \right\}
\end{align*}
where $x = \frac{y}{|y|^2}$, $y \in \Omega^* \setminus \{ 0 \}$.
Thus
\begin{align*}
	\nabla_y w(y) \Big|_{y \in \pd\Omega^*} &= \left[ |x|^N \left\{ (2-N) \underbrace{u(x)}_{=0} \frac{x}{|x|^2} - 2 (\nabla_x u(x) \cdot x) \frac{x}{|x|^2} + \nabla_x u(x) \right\} \right]_{x \in \pd\Omega} \\
	&= |x|^N \( - 2 (\nabla_x u(x) \cdot \frac{x}{|x|}) \frac{x}{|x|} + \nabla_x u(x) \)_{x \in \pd\Omega} \\
	& = |x|^N \( \vec{a}(x) - \vec{b}(x) \),
\end{align*}
where
\[
	\vec{a}(x) = \nabla_x u(x) - (\nabla_x u(x) \cdot \frac{x}{|x|}) \frac{x}{|x|}, 
	\quad \vec{b}(x) = (\nabla_x u(x) \cdot \frac{x}{|x|}) \frac{x}{|x|}.
\] 
Note that 
\[
	\vec{a}(x) + \vec{b}(x) = \nabla_x u(x),  \quad \vec{a}(x) \perp \vec{b}(x) \quad \text{for} \ x \in \pd\Omega, 
\]
thus $|\vec{a}(x)|^2 + |\vec{b}(x)|^2 = |\nabla_x u(x)|^2$. %by Pythagorean theorem. 
Therefore, we have
\begin{align*}
	|\nabla_y w(y)|^2 \Big|_{y \in \pd\Omega^*} 
% = |x|^{2N} \( |\vec{a}(x)|^2 + |\vec{b}(x)|^2 \) \Big|_{x \in \pd\Omega} 
= |x|^{2N} |\nabla_x u(x)|^2 \Big|_{x \in \pd\Omega}.
\end{align*}
Since $u$ is a solution of \eqref{OD_E}, we see that $w$ in \eqref{Kelvin} satisfies
\begin{equation}
\label{OD_Serrin}
	\begin{cases}
	&\Delta w = N \ \text{in} \ \Omega^* \setminus \{ 0 \}, \\
	&w = 0 \quad \text{on} \ \pd \Omega^*, \\
	&|\nabla w| = c > 0 \quad \text{on} \ \pd \Omega^*.
	\end{cases}
\end{equation}
Furthermore, since $u$ satisfies $\lim_{|x| \to \infty} u(x) = 0$, its Kelven transform $w$ satisfies
\[
	\lim_{|y| \to 0} \frac{|w(y)|}{|y|^{2-N}} = 0.
\]
Thus $y = 0$ is a removable singularity and $w$ satisfies the equation on the whole $\Omega^*$. 
Then Serrin's result assures that $\Omega^*$ must be a ball and $w$ must be radially symmetric with respect to the center of the ball.
This implies that $\Omega$ must be the compliment of the ball and $u$ is radially symmetric with respect to the same point.
If $\Omega = B^c$, then $\Omega^* = {\rm int }(B)$ and
\[
	w(y) = \frac{|y|^2-1}{2} \quad (y \in {\rm int }(B))
\]
is the unique solution of \eqref{OD_Serrin} with $c = 1$.
Then by \eqref{Kelvin}, we have
\[
	u(x) = |y|^{N-2} w(y) = \frac{|y|^N - |y|^{N-2}}{2} = \frac{|x|^{-N} - |x|^{2-N}}{2}.
\]

When $N = 2$, the same computation also holds for the Kelvin transform $w(y) = u(\frac{y}{|y|^2})$.
Moreover, we can claim that $y=0$ is again removable since in this case  
\[
	\lim_{|y| \to 0} \frac{|w(y)|}{\log (1/|y|)} = 0
\]
follows 
from the weaker condition at $\infty$: $u(x) = O(1)$ as $|x| \to \infty$.
We have finished the proof of Proposition \ref{Prop:E}.
\end{proof}

\vspace{1em}\noindent
{\it Proof of Theorem \ref{Theorem:E}.} 
Once the uniqueness of the explicit solution $U$ in \eqref{U} to the problem \eqref{OD_E} is proven in Proposition \ref{Prop:E},
the proof of Theorem \ref{Theorem:E} can be done in exactly the same way as the proof of Theorem \ref{Theorem:H}.
\qed

\vskip 0.5cm

\noindent\textbf{Acknowledgement.} 
The second author (F.T.) was supported by JSPS Grant-in-Aid for Scientific Research (B), No. 23K25781, 
and was partly supported by Osaka Central University Advanced Mathematical Institute (MEXT Joint Usage/Research Center on Mathematics and Theoretical Physics).

\end{document}